\documentclass[final ]{IEEEtran}

 \usepackage{hyperref}  

 \usepackage{datetime}  
\usepackage{nicematrix}
\usepackage{multicol}
\usepackage{mathtools}
\usepackage{url}
\usepackage[makeroom]{cancel}
\usepackage{soul}

\usepackage{amsmath,amssymb}
\usepackage{amsthm,thmtools}
\usepackage{amsfonts}
\usepackage{blkarray}
\usepackage{circuitikz}
\usepackage{bigstrut}
\usepackage{tikz}
\usetikzlibrary{tikzmark,calc,fit}
\usepackage{pgfplots}

 \usepackage{tkz-graph}

 \usepackage{graphicx}
\usepackage{enumerate}
\usepackage{color}      

\usepackage{verbatim}

\usepackage{amssymb}
\usepackage{amsbsy}
\usepackage{amsmath,lipsum}
\usepackage{mathabx}
\usepackage{setspace}
\usepackage{url}

\usepackage{float}
 \usepackage{dsfont}

 \newcommand{\Spec}{\operatorname{Spec}}

 \newcommand{\real}{\operatorname{Re}}
 \newcommand{\imag}{\operatorname{Im}}
 \newcommand{\diag}{\operatorname{diag}}
 \newcommand{\sign}{\operatorname{sgn}}

\newtheorem{theorem}{Theorem}
\newtheorem{assumption}{Assumption}

\newtheorem{proposition}{Proposition}

\newtheorem{problem}{Problem}

\newtheorem{example}{Example}
\newtheorem{remark}{Remark}

\newcommand {\R}{\mathbb R}

\newcommand {\C}{\mathbb C}

\newcommand {\T}{\mathbb T^n}

\newcommand{\be}{\begin{equation}}
\newcommand{\ee}{\end{equation}}

\newcommand{\Z}{\mathbb Z}

\usepackage{xcolor}
\usepackage{soul}

\renewcommand{\Re}{\operatorname{Re}}

\title{\LARGE \bf
An application of the  mean motion problem to  
time-optimal control}

\renewcommand{\Re}{\operatorname{Re}}

\author{Omri Dalin, Alexander Ovseevich and Michael Margaliot\thanks{The authors are with the School of Elec. Eng., Tel Aviv University, Israel~6997801. Correspondence: michaelm@tauex.tau.ac.il \\
This research is partially supported by a research grant form the Israeli Science Foundation~(ISF).} }

\begin{document}
\maketitle

\begin{abstract}
We consider  time-optimal controls of a  controllable linear  system   with a scalar control on a  long  time interval. It is well-known that if all the eigenvalues of the matrix describing the linear system dynamics 
are real then any time-optimal control has a bounded number of
switching  points, where the bound
does not depend on the length of the time interval.
We consider  the case where the governing matrix   has purely imaginary  eigenvalues, and show that then, in the generic case, the number of switching points is  bounded from below by a linear function of the length of the time
interval.   The    proof  is based on   relating the    switching function in the optimal control problem to
the   mean motion problem that dates back to Lagrange and was solved  by Hermann Weyl.
\end{abstract}


\section{Introduction}
Consider the single-input
linear control system
\begin{align}\label{eq:lin_cont_sys}
\dot x&=Ax + b u,
\end{align}
 with~$x:[0,\infty)\to \R^n$, $A\in\R^{n\times n}$, $b\in\R^n$,
 and~$u:[0,\infty)\to[-1,1]$.
We assume throughout that the system is controllable~\cite{sontag_book_control_theory}.   Fix arbitrary~$p,q\in\R^n$, and consider the problem of finding a measurable  control~$u$, taking values in~$[-1,1]$ for all~$t\geq 0$, that steers the system from~$x(0)=p$ to~$x(T)=q$ in minimal time~$T$.

It is well-known that such a control 
exists and satisfies~$u(t)=\sign (m(t))$,
where  the switching function $m:[0,T]\to\R$ is given by
\begin{equation}\label{switching}
 m(t)=p^\top (t)b,
\end{equation}
where the ``adjoint state vector'' $p:[0,T]\to\R^n\setminus\{0\}$   of the Pontryagin maximum principle satisfies
\begin{equation}\label{adjoint}\dot p(t)=- A^\top p(t).
\end{equation}
Under the controllability assumption, $m(t)$ has a finite number of zeros on any open  time interval, so 
  a time-optimal control is  ``bang-bang'',  with switching points at isolated
  time instants~$t_i$ such that~$m(t_i)=0$.

An interesting  question is to determine     the number of switching points   in the interval~$[0,T]$, denoted~$N(T)$. This   has important implications. For example,
it is well-known~\cite[Theorem~6-8]{athans_falb} that if all the eigenvalues of~$A$  are real then~$N(T)\leq n-1$ for all~$T>0$,  implying that the solution of the time-optimal control problem reduces to the finite-dimensional problem of  determining   up to~$n-1$ values~$t_1,\dots,t_{n-1}$.  Furthermore,   bounding the number of switching points in time-optimal controls
is also important for nice reachability results, that is, finding a set of controls~$\mathcal{W}$ with ``nice'' properties such that the  problem of steering the system between two points can always  be solved using a control from~$\mathcal W$ (see, e.g.~\cite{suss_nice_reach,levinson_select,SHARON_3_nilp}).

Typically the value of the adjoint state vector  is   not known explicitly, so it is natural to      study an ``abstract'' 
switching
function
\be\label{eq:abst_m}
m(t;p,b, A):=p^\top e^{-At}b
\ee
with~$p,b\in\R^n\setminus\{0\}$ and~$A\in\R^{n\times n}$. 
 The zeroes of the switching function are called the switching points.
 
 Here, we consider the case where all  the eigenvalues of the matrix~$A$ are  purely imaginary. Our main contributions include the following: 
 \begin{itemize}
     \item We develop a new approach for analyzing  the zeros of~$m(t)$ using the classical, yet generally forgotten, problem of mean motion that was solved  by Hermann Weyl in 1938~\cite{Weyl_meanmotion};
     \item  Using this new approach we show that 
   generically  the number~$N(T)$ of switching   points on the interval~$[0,T]$ 
satisfies
\begin{equation}\label{asymp_in}
N(T)\geq c T \text{ for all sufficiently large } T,
\end{equation}
where~$c$ is a positive constant. Our approach also
 provides a   closed-form expression
for~$c$ in terms of integrals of Bessel functions.
 \end{itemize}

The next section  reviews the mean motion problem  including the  Bohl–Weyl–Wintner (BWW) formula that allows to derive an explicit expression for~$c$.
For the sake of completeness,   we include more details 
on the proof  of this formula in the Appendix.
Section~\ref{sec:density}
 states and proves  the main result. 
 The final section concludes, and describes some directions for further research.


We use standard notation.
$\R^n$ [$\C^n$]
is the~$n$-dimensional vector space over the field of real [complex] scalars, and we abbreviate~$\R^1$ to~$\R$ [$\C^1$ to~$\C$].
For a complex
number~$z\in\C$, $|z|$ is the absolute value of~$z$,
and~$ \arg (z)$ is the argument of~$z$, so the polar representation of~$z$ is
$z=|z| e^{i\arg (z)}$. Also,
  $\real(z)$  [$\imag(z)$] denotes the  real [imaginary] part of~$z$,  and~$\bar z=|z| e^{-i\arg (z)}$ is the complex conjugate   of~$z$.
Vectors [matrices] are denoted by small [capital] letters.
The transpose of a matrix~$A$ is~$A^\top$.

\section{Preliminaries}
In this section, we
review several results
from the theory of mean motion that will be useful for the analysis of time-optimal controls.

\subsection{The problem of mean motion}
The mean motion problem
goes back to Lagrange's analysis of the secular perturbations of the major planets.
It provides a way to calculate the average rate at which an object orbits a central body.

Consider the complex function~$z:[0,\infty)\to \C$ defined by
  \begin{equation}\label{eq:rotating}
z(t) : =\sum_{k=1}^n a_k e^{i(\lambda_k t+\mu_k)} ,
\end{equation}
where~$a_k,\lambda_k,\mu_k\in \R$.
Note that this may be interpreted as a weighted sum of linear oscillators with different frequencies. 

Assume that~$z(t)\not =0$ for all~$t\geq 0$.
Then~$\Phi(t):=\arg( z(t))$ is a continuous function.  
\begin{problem}[mean motion]\label{prob:mean_motion}
Determine whether the asymptotic angular velocity of~$z(t)$, that is, the   limit  
 \be\label{eq:def_omeha}
 \Omega:=\lim_{t\to \infty}\frac{\Phi(t)}{t},
 \ee
   exists, and if so,   find  its value.
\end{problem}

\begin{example}
    Suppose that all 
    the~$\lambda_k$s are equal, and we denote their common value by~$\lambda$, then
\begin{align*}
z(t)&=\sum_{k=1}^n a_ke^{i (\lambda t+\mu_k) }\\
&= e^{i\lambda t} \sum_{k=1}^n a_k e^{i \mu_k },
\end{align*}
so
\begin{align*}
    \Phi(t)&= \arg( z(t))\\&= \lambda t+\arg(\sum_{k=1}^n a_k e^{i \mu_k }) ,    
\end{align*}
and~$\Omega=\lim_{t\to \infty}\frac{\Phi(t)}{t}=\lambda$.
\end{example}

\subsection{Bessel functions}
Bessel functions are ubiquitous in mathematics and physics~\cite{watson}.
The~$p$-order Bessel function~$J_p:\C\to \C$  is defined by 
\be\label{eq:p-bessel}
J_p(x) := \frac {(x/2)^p}{  \Gamma(p+\frac1{2}) \sqrt{\pi}}\int_0^\pi \sin^{2p}(\phi)e^{-ix\cos(\phi)} d\phi,
\ee
where~$\Gamma$ is the Gamma function. In the mean motion problem,  we encounter two Bessel functions: 
\[
J_0(x) =  \frac1{\pi}  \int_0^\pi  e^{-ix\cos(\phi)} d\phi,
\]
and
\[
J_1(x) = \frac { x   }{   \pi} \int_0^\pi \sin^{2 }(\phi)e^{-ix\cos(\phi)} d\phi.
\]
These functions appear 
through their connection to   
$n$-dimensional balls and spheres. To explain this,
let~$1_{B^n}:\R^n\to\{0,1\}$ denote the  indicator function of the $n$-dimensional ball, that is,~$1_{B^n}(x)=1$ if~$|x|\leq 1$ and~$1_{B^n}(x)=0$, otherwise.  The Fourier transform of~$1_{B^n}$  is
\begin{align}\label{eq:four_ball}
\hat 1_{B^n}(\xi)&= \frac1{(2\pi)^{n/2}}
\int_{|x|\leq 1} e^{i x^\top  \xi } dx_1\dots dx_n\nonumber\\& = \frac1{(2\pi)^{n/2}}
\int_{|x|\leq 1} e^{-i x^\top  \xi } dx_1\dots dx_n . 
\end{align}
Since~$1_{B^n}(x)$ is a radial function (that is, it only depends only on~$|x|$),
$\hat 1_{B^n}(\xi)$ will   depend on~$|\xi|$ (see, e.g.~\cite[Chapter~IV]{Four_intro}) and it is
  enough  to consider the particular
  case where~$\xi=\begin{bmatrix}
0&\dots&0&r
\end{bmatrix}^\top$,
with~$r>0$. Then,
\begin{align*}
  \hat 1_{B^n}(\xi)&
 = \frac1{(2\pi)^{n/2}} \int_{|x|\leq 1}  e^{-i x_n r} dx_1 \dots dx_n\\
  &=  \frac { V_{n-1}}{(2\pi)^{n/2}}
  \int_{-1}^1 (\sqrt{ 1-x_n^2})^{n-1 }  e^{-i x_n r} dx_n,
\end{align*}
where~$V_{n-1} $ is the volume of the~$(n-1)$-dimensional unit ball.
Setting~$x_n= \cos (\phi)$ and using the fact  that~$V_d=\frac{\pi^{d/2}} { \Gamma(1+(d/2)) } $ gives
   \begin{align*}
  \hat 1_{B^n}(\xi)
  &=  \frac{ 1}{ 2 ^{n/2} \Gamma(1+((n-1)/2))\sqrt{\pi}}
 \\& \times  \int_{0}^\pi \sin^{n}(\phi)    e^{-i \cos(\phi)  r} d \phi ,
\end{align*}
and comparing this with~\eqref{eq:p-bessel} yields
\[ 
 \hat 1_{B^n}(\xi)  = |\xi|^{-n/2}J_{n/2}(|\xi|).
 \]
In particular, for the special case where~$n=2$, we have
\be\label{eq:ind_func_B2}
 \hat 1_{B^2}(\xi)  = |\xi|^{-1}J_{1}(|\xi|).
 \ee

\subsection{The Bohl--Weyl--Wintner formula}
The Bohl--Weyl--Wintner~(BWW) formula~\cite{Wintner-1933} 
gives a closed-form expression in terms of integrals of Bessel functions 
for a   probability function defined on an~$n$-dimensional torus $\T=(\R/2\pi\Z)^n$, that is, the direct product of~$n$ circles.  Let~$\phi_k\in \R/2\pi\Z$, $k=1,\dots,n$, be the angular coordinates on the torus.
Fix~$n$ complex numbers~$a_1,\dots,a_n\in\C$. Associate with every point~$(\phi_1,\dots,\phi_n)\in\T$
  a complex number
\[
z(\phi_1,\dots,\phi_n):=\sum _{k=1}^n a_ke^{i\phi_k}.
\]
Geometrically, this can be interpreted as
the position of the end-point of a multi-link robot arm  where the $k$th  link has length~$|a_k|$, and the
angle between link~$k$
and link~$(k+1)$ is~$\phi_k $ (see Fig.~\ref{fig:viewz}).

\begin{figure} 
\begin{tikzpicture}[scale=1.5]
    \def\aone{2}      
    \def\phiOne{30}   
    \def\atwo{1.5}    
    \def\phiTwo{45}   
    \def\phiExtend{27.5} 

    \pgfmathsetmacro{\Ax}{\aone * cos(\phiOne)}
    \pgfmathsetmacro{\Ay}{\aone * sin(\phiOne)}
    \pgfmathsetmacro{\Bx}{\Ax + \atwo * cos(\phiOne + \phiTwo)}
    \pgfmathsetmacro{\By}{\Ay + \atwo * sin(\phiOne + \phiTwo)}

    \pgfmathsetmacro{\PhiZ}{atan2(\By, \Bx)}

    \coordinate (O) at (0,0);
    \coordinate (A1) at (\Ax,\Ay);
    \coordinate (A2) at (\Bx,\By);

    \draw[->] (-0.5,0) -- (2.5,0) node[below] {$\mathrm{Re}$};
    \draw[->] (0,-0.5) -- (0,2.5) node[left] {$\mathrm{Im}$};

    \draw[->,blue,thick] (O) -- (A1) node[midway, above  ] {\footnotesize $a_1$};
    \draw[->,red,thick] (A1) -- (A2) node[midway, below right] {\footnotesize $a_2$};

    \draw[gray, thick] (O) -- (A2) node[midway, above left] {\footnotesize $|z(\phi)|$};

    \draw[dashed] (A1) --++ (1.2,0);

    \draw[green,thick] (0.4,0) arc (0:\phiOne:0.4);
    \node[green] at (0.5,0.15) {\footnotesize $\phi_1$};

    \draw[purple,thick] ($(A1)+(0.3,0)$) arc (0:\phiTwo+\phiExtend:0.3);
    \node[purple] at ($(A1)+(0.4,0.2)$) {\footnotesize $\phi_2$};

    \draw[orange,thick] (0.7,0) arc (0:\PhiZ:0.7);
    \node[orange] at ({1.1*cos(\PhiZ/2)}, {1*sin(\PhiZ/2) - 0.20}) {\footnotesize $\;\;\;\;\arg(z(\phi))$};


\end{tikzpicture}
\caption{Geometric representation of~$z(\phi)=z(\phi_1,\phi_2)$.\label{fig:viewz}}
\end{figure}
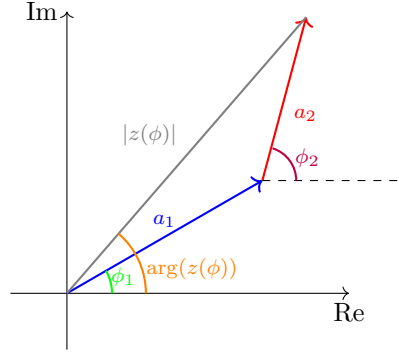

The canonical (Haar)  probability measure on the torus is
\begin{equation}\label{haar}
d\mu=\frac1{(2\pi)^n} d\phi_1\dots d\phi_n.
\end{equation}
Let~$W_n(r)=W_n(r;a_1,\dots,a_n)$ denote the probability 
that~$|z|\leq r$, where $z=z(\phi_1,\dots,\phi_n)$. Thus,
\begin{align}\label{eq:W_n(r)}   W_n(r) =
\int_{ |z|\leq r }d\mu
\end{align}
is the ``volume'' of all the angles in the $n$-dimensional
 torus
yielding~$|z|=|\sum _{k=1}^n a_ke^{i\phi_k}|\leq r$.

For small values of~$n$, $W_n(r)$ can be computed explicitly using geometric arguments. 
The next example demonstrates this. 

 \begin{example}
Consider the case~$n=2$. Fix~$a_1,a_2>0$. 
We compute~$W_2(r)=W_2(r;a_1,a_2)$ for any~$0\leq r\leq a_1+a_2$. 
To do so, we first determine 
all the angles~$0\leq\phi_1,\;\phi_2< 2\pi$ such that~$|z|=|a_1e^{i\phi_1}+a_2e^{i\phi_2}|\leq r$.
Since the calculation of~$W_2(r)$ is 
  invariant with respect to rotations, we  may assume that~$\phi_1=0$ (see Fig.~\ref{fig:two_d}).
  \begin{figure}[H]
 \begin{center}
  \includegraphics[scale=0.6]{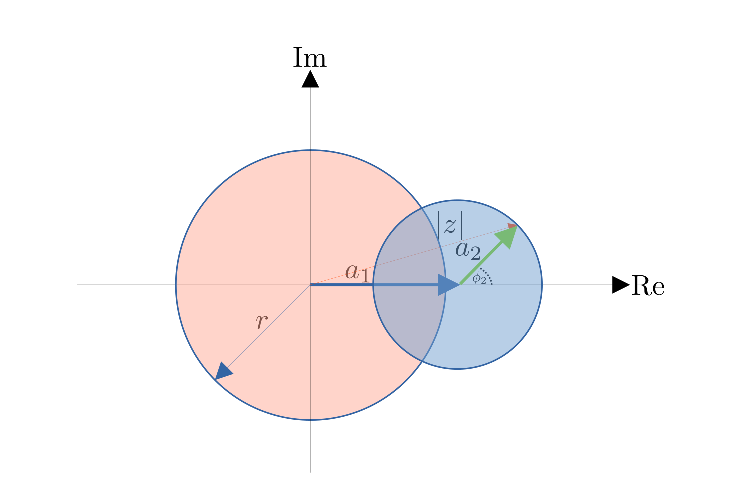}
  \caption{Plotting $z=a_1e^{i\phi_1}+a_2e^{i\phi_2}$, with~$\phi_1=0$, in the complex plane.}
  \label{fig:two_d}
\end{center}
\end{figure}
  We have
\begin{align*}
    |z|^2&=\big(a_1+a_2\cos(\phi_2)\big)^2+\big(a_2\sin(\phi_2)\big)^2\\
    &=a_1^2+2a_1a_2\cos(\phi_2)+a_2^2 .
 \end{align*}
   Let  
\begin{align}\label{eq:def_q}
    q(r;a_1,a_2):=\frac{r^2-(a_1^2+a_2^2)}{2a_1a_2}.
\end{align}
Then~$|z|^2\leq r^2$ iff
\begin{align*}
    \cos(\phi_2)\leq q ,
\end{align*}
that is, iff 
\begin{align*}
    \arccos (q) \leq\phi_2\leq 2\pi-\arccos(q).
\end{align*}
 Eq.~\eqref{eq:W_n(r)} now gives 
\begin{align}\label{eq:W_2(r)}
    W_2(r)&=\frac{1}{(2\pi)^2}\int_0^{2\pi}d\phi_1\int_{\arccos(q)}^{2\pi-\arccos(q)}d\phi_2\nonumber\\
    &=\frac{1}{2\pi}\big(2\pi-2\arccos( q) \big)\nonumber\\
    &=1-\frac{\arccos(q)}{\pi}.
\end{align}
Note that~$ W_2(r)\in[0,1]  $, as expected.  If~$r=a_1+a_2$  then~$q=1$  and~\eqref{eq:W_2(r)} gives~$W_2(r)=1$.
If~$r=|a_1-a_2|$  then~$q=-1$  and~\eqref{eq:W_2(r)} gives~$W_2(r)=0$ (see Fig.~\ref{fig:two_d}). 

\end{example}

The BBW formula   asserts  that \begin{equation}\label{BWW}
W_n(r;a_1,\dots,a_n)=r\int_0^\infty J_1(r\rho)\prod_{k=1}^n J_0(|a_k|\rho)d\rho,
\end{equation}
where $J_0,\,J_1$ are Bessel functions. Note that this
reduces the computation of the $n$-dimensional integral for~$W_n$ in~\eqref{eq:W_n(r)}
to a one-dimensional integral. 

For the sake of completeness, we include a self-contained
proof of~\eqref{BWW}
 in the Appendix.

\subsection{Solution of the mean motion problem}
Returning to Problem~\ref{prob:mean_motion}, 
  we say that the $\lambda_k$s in~\eqref{eq:rotating}
  are
  \emph{non-resonant}
  if there  are no non-trivial relations of the form 
  \be\label{eq:non_rse}
  \sum_{k=1}^n  \lambda_k \ell_k=0, \text{ where every }
\ell_k \text{ is an  integer}.
\ee
 Hermann Weyl \cite{Weyl_meanmotion} proved that under this condition the mean motion~$\Omega$ in~\eqref{eq:def_omeha}   exists, and satisfies
  \begin{equation}\label{mean}
\Omega=\sum_{k=1}^n \lambda_k V_k ,
\end{equation}
where
\be\label{eq:vks}
V_k : =W_{n-1}(a_k;a_1,\dots,a_{k-1},a_{k+1},\dots,a_n).
\ee
Furthermore,
the  $V_k$s  are non-negative, and $\sum_{k=1}^n V_k=1$~\cite{Weyl_meanmotion}.
In particular, if~$\lambda_k>0$ for all~$k$ then~$\Omega>0$, and if   all the~$\lambda_k$s are equal to a common value~$\lambda$ then Eq.~\eqref{mean} gives
\[ \Omega=\sum_{k=1}^n \lambda V_k=\lambda.
\]

Thus, the mean motion~$\Omega$ is a weighted  average of  the angular velocities $\lambda_k$.
The weight~$V_k$
corresponding to~$\lambda_k$
is  the ``volume''  of the set of points~$(\phi_1,\dots,\phi_{k-1},\phi_{k},\dots,\phi_n)$
in an~$(n-1)$-dimensional torus
such that~$   | {\displaystyle
\sum _{   
\ell \not = k}  }
a_\ell e^{i\phi_\ell} |\leq |a_k|$.

\begin{example}
    Consider the case~$n=2$, that is,
    \[
    z(t)=a_1 e^{i(\lambda_1 t+ \mu_1) } +a_2 e^{ i(\lambda_2 t+ \mu_2) },
    \]
    and assume that~$|a_1|>|a_2|$. In this case, $a_2 e^{i(\lambda_2 t+ \mu_2) }$ is a ``small perturbation'' added to~$a_1 e^{i(\lambda_1 t+ \mu_1) }$.
  Eq.~\eqref{mean} gives~$\Omega=\lambda_1 V_1+\lambda_2 V_2$, where~$V_1$ is the
  probability that~$|a_2e^{i \phi}|\leq a_1$, that is,~$V_1=1$ (and similarly~$V_2=0$). Thus, in this case~$\Omega=\lambda_1$. 
\end{example}

 Note that combining
 Eqs.~\eqref{mean}, \eqref{eq:vks},
 and the BBW-formula
 yields
 a closed-form  expression for the mean motion~$\Omega$
 in terms of  integrals of  Bessel functions.

\begin{example}
     Consider the case~$n=3$, $a_1= 1$,
     $a_2 =2.5$, $a_3=3$, $\lambda_1=\sqrt{2}$,
     $\lambda_2=3 $,~$\lambda_3=\sqrt{3}$
      and~$\mu_i=0$, $i=1,2,3$, so 
\begin{align*}
    z(t)=  e^{i \sqrt{2} t }+2.5 e^{i 3 t }+3  e^{i  \sqrt{3} t }.
\end{align*}
Then
\begin{align*}
    \Omega &= \sum_{i=1}^3 \lambda_i V_i\\
    &= \sqrt{2} W_2(1; 2.5,3)+3 W_2(2.5; 1,3)+\sqrt{3} W_2(3; 1,2.5)\\
    &= \sqrt{2}   \int_0^\infty J_1( \rho) 
    J_0( 2.5 \rho)J_0( 3 \rho)d\rho\\
    &+ 7.5   \int_0^\infty J_1( 2.5 \rho) 
    J_0(  \rho)J_0( 3 \rho)d\rho\\
    &+3\sqrt{3}   \int_0^\infty J_1( 3 \rho) 
    J_0(  \rho)J_0( 2.5 \rho)d\rho.
\end{align*}
Numerical integration  using Matlab gives
\[
\Omega=2.0614
\]
(to 4-digit accuracy). Fig.~\ref{fig:bessel}
 shows the value~$\frac{\arg(z(T)) }{T}$ as a function of~$T$, and it may be seen that this indeed converges to~$\Omega$ as~$T\to\infty$. 
\end{example}

\begin{figure}[t]
\centering
\includegraphics[scale=0.6]{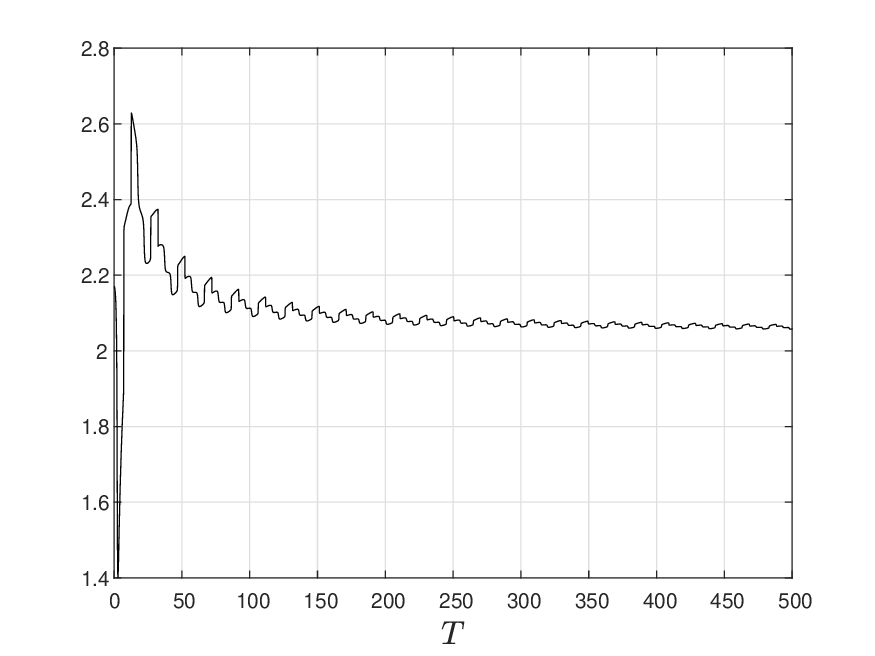}
\caption{Plotting~$\frac{\arg(z(T)) }{T}$ as a function of~$T$. \label{fig:bessel}}
\end{figure}

\begin{remark}
    The BBW formula is a solution to a ``static'' problem of determining a ``volume'' of angles in the torus that satisfies a certain  inequality. The mean motion problem is ``dynamic'',  as it considers the asymptotic behaviour of the time-dependent function~$z(t)$. For the sake of completeness, we give an informal explanation on how these two problems are related. 
Since~$z(t)=|z(t)|e^{i\Phi(t)}$, we have~$\Phi(t)=\imag(\log(z(t)))$, so 
\begin{align*}
    \dot \Phi(t) &= \imag (\frac{\dot z(t)}{z(t)})\\
    &=\imag (  \frac{\sum a_k\lambda_k   i e^{i(\lambda_k t +\mu_k)}}{\sum a_k e^{ i(\lambda_k t+\mu_k)}}  ) \\
    &=\real (  \frac{\sum a_k\lambda_k     e^{i(\lambda_k t +\mu_k)}}{\sum a_k e^{ i(\lambda_k t+\mu_k)}}  ) .
\end{align*}
This implies that we can view~$\dot\Phi(t)$ as a mapping from~$(\lambda_1 t,\dots,\lambda_n t)\in\T$
to~$\R$. Now, 
\begin{align}\label{eq:intg_rih}
\Omega&=\lim_{t\to\infty}\frac {\Phi(t)-\Phi(0)}{t}\nonumber\\
&=\lim_{t\to\infty}
\frac1{t}\int_0^t \dot \Phi(s)ds,
\end{align}
  and the non-resonance condition implies that we can calculate this  integral 
using  a ``static'' integral over the torus. 
\end{remark}

\begin{remark}
Eq.~\eqref{mean} holds even if condition~\eqref{eq:non_rse} is relaxed to  the following {\em almost non-resonancy} condition: there  are no non-trivial relations of the form $\sum_{\ell=1}^n \lambda_k \ell_k=0$, where
the~$\ell_k$s are integers, and $\sum_{k=1}^n \ell_k=0$. If~$\lambda_1=\dots=\lambda_n$
then the non-resonance condition fails, but the almost non-resonance condition holds. 
\end{remark}

\begin{remark}
Writing~\eqref{eq:rotating}  as
\[
z(t)   =\sum_{k=1}^n (a_k e^{i \mu_k} ) e^{i \lambda_k t  } ,
\]
and noting that~$\Omega$ does not depend on the~$\mu_k$s 
implies that we can actually assume that
the~$a_k$s in~\eqref{eq:rotating}  
are complex numbers,
and the formula for~$\Omega$ will depend
only on  the absolute values of the~$a_k$s. 
\end{remark}
\section{Main result}\label{sec:density}

If the spectrum of~$A $ in~\eqref{eq:lin_cont_sys} is real then it is well-known (see, e.g., \cite[Theorem~6-8]{athans_falb})  that~$N(T)\leq n-1$ for any~$T > 0$. We consider the    case where  all the eigenvalues of~$A$ are purely imaginary. 
 
\begin{assumption}\label{assump:dominanat}
The dimension~$n$ is even and the 
eigenvalues of~$A\in\R^{n\times n}$ are 
\[
i\lambda_1, -i\lambda_1 ,\dots, i\lambda_{n/2}, -i\lambda_{n/2}, 
\]
where all the~$\lambda_k$s are real. 
\end{assumption}

We can now state our main result. 

\begin{theorem}\label{thm:main}
Suppose that~$A$ satisfies Assumption~\ref{assump:dominanat} and that the pair~$(A,b) $ is controllable. 
 Then, for   generic vectors~$p,b\in\R^{n}\setminus\{0\}$
 there exists~$c>0$ such that 
   for any~$T$ large enough  the number of zeros of the switching function~$m$ in~\eqref{eq:abst_m} on the interval~$[0,T]$ satisfies
   \[
   N(T)\geq c T+o(T).
   \]
\end{theorem}

The proof of this result, given below, can actually be used to
provide a closed-form
expression for~$c$ in terms of a  solution to 
the mean motion problem. 

\begin{IEEEproof}
We may assume that there exists a non-singular matrix~$H\in\R^{n\times n}$ such that~$A=H^{-1}\diag(B_1,\dots,B_{n/2})H$, with
\[
B_i:=\begin{bmatrix} 0& \lambda_i\\-\lambda_i &0 \end{bmatrix}.
\]
The switching function~\eqref{eq:abst_m}  
becomes 
\begin{align*} 
m(t) &= p^\top  H^{-1} \diag(\exp(B_1 t),\dots,\exp(B_{n/2}t)) H b\\
&=  \sum_{k=1}^{n/2}  c_k \cos(\lambda_k t)+d_k \sin(\lambda_k t),
\end{align*}
where~$c_k,d_k\in\R$   depend on  the 
entries     of~$b$, $p$, and~$H$. Thus, 
\begin{align*} 
m(t) &= \real( \sum_{k=1}^{n/2}  c_k e^{i \lambda_k t}+
d_k e^{i(    \lambda_k t -(\pi/2) )}  )
\\
&=  \real(  \sum_{k=1}^{n/2}   a_k  e^{i \lambda_k t}   ) ,
\end{align*}
with~$a_k:=c_k+d_ke^{-i\pi/2}=c_k-id_k$. 
Define 
\[
z(t):= \sum_{k=1}^{n/2}   a_k  e^{i \lambda_k t} .
\]
 Every time~$t_i\geq 0 $ such that~$ \arg(z(t_i)) = \pi/2$ or~$\arg(z(t_i)) =3  \pi/2  $
is a zero of the swicthing function~$m(t)$. Consider the asymptotic
behavior of~$\arg(z(t))$ on the interval~$ [0,T]$, with~$T$  large.  
The solution of the mean motion problem implies that
\[\
\Omega : = \lim_{T\to\infty}  \frac{ \arg(z(T))-\arg(z(0))} {T} 
\]
exists and satisfies~\eqref{mean}. Thus, for~$T$ large enough  the number of zeros of~$m$ on the interval~$[0,T]$
is bounded from below by
\[
\frac{|\Omega|}{\pi}T+o(T) , 
\]
and this completes the proof of Theorem~\ref{thm:main}. 
\end{IEEEproof}

\begin{example}
Consider  the case where~$n=4$,
\[
A=\begin{bmatrix}
    0 & \zeta_1 & 0 &0 \\
    -\zeta_1 &0 &0 &0\\
    0&0&0&\zeta_2\\
    0 & 0& -\zeta_2 &0 
\end{bmatrix},
\]
with~$\zeta_1,\zeta_2\not =0$,  
$\zeta_1>\zeta_2$,
and~$b=\begin{bmatrix}
   0& 1&0&1
\end{bmatrix}^\top$. 
Thus,~$A$ has distinct
purely imaginary eigenvalues:~$i\zeta_1,-i\zeta_1,i\zeta_2,-i\zeta_2$. 
A calculation shows that  
\begin{align*}
\det \begin{bmatrix}
    b& Ab & A^2b&A^3b
\end{bmatrix} &=\zeta_1 \zeta_2 (\zeta_1^2-\zeta_2^2)^2\\&\not=0,
\end{align*}
so the pair~$(A,b)$ is controllable. 
System~\eqref{eq:lin_cont_sys} becomes
\begin{align}\label{oscillators}
\ddot x_1 &+\zeta_1^2 x_1=\zeta_1 u,\nonumber\\
\ddot x_3&+\zeta_2^2 x_3=\zeta_2 u. \nonumber
\end{align}
This describes  the motion of two  linear oscillators (e.g., springs)
coupled  by a common  forcing term~$u(t) $ (and in the optimal control problem we assume that~$|u(t)|\leq 1$ for all~$t$).
In this case, 
\begin{align*}
  m(t)&=p^\top  e^{-At} b \\
    &= -p_1\sin(\zeta_1 t) -p_3\sin(\zeta_2 t) 
     +p_2 \cos(\zeta_1 t) +p_4 \cos(\zeta_2 t) .
\end{align*}
This can    be expressed as
\begin{align*}
     m(t)&= \real ( - p_1 e^{i(\zeta_1  t-(\pi/2))}  -p_3 e^{i(\zeta_2 t-(\pi/2))}\\&+p_2  e^{i \zeta_1 t}+p_4
  e^{i \zeta_2 t} )   \\
  &=\real (a_1 e^{i\zeta_1 t} +  a_2 e^{i\zeta_2 t}
  ) ,
\end{align*}
 with~$a_1:=-p_1 e^{-i\pi/2}+p_2$ and~$a_2:=-p_3 e^{-i\pi/2} +p_4 $.    
\end{example}

For low values of~$n$  it is possible to give more accurate bounds
for~$N(T)$ and show that they agree with the asymptotic lower bound in Theorem~\ref{thm:main}.
The next result  demonstrates this.

\begin{proposition}\label{prop:two_oscillators}
   Consider the case where
    \begin{align*} 
       m(t)=\real(a_1e^{i\lambda_1 t}+a_2e^{i\lambda_2 t}), 
    \end{align*}
    with~$a_1,a_2\in\R\setminus\{0\}$,
  and  assume, without loss of generality, that
    \[
    \lambda_1>\lambda_2.
    \]
  If~$a_1>a_2$ then
     \begin{align} \label{eq:a1la2}
        \frac{\lambda_1}{\pi}T\leq N(T)\leq\frac{\lambda_1}{\pi}T+1,\text{ for all }T>0.
    \end{align}
If~$a_1<a_2$ then 
 \begin{align} \label{eq:seclow}
        \frac{\lambda_2}{\pi}T\leq N(T)\leq\frac{\lambda_1}{\pi}T,\text{ for all }T>0.
    \end{align}
    \end{proposition}
    \begin{IEEEproof}
A time~$t_i$ is zero of~$m(t)$ iff
\[
 \cos(\lambda_2 t_i) =-\frac{a_1}{a_2} \cos(\lambda_1 t_i).
\]
Suppose that~$a_1>a_2$. 
Fix an integer~$\ell\geq 0$, and consider the time interval
\[
L_1:=[\frac{2\pi\ell}{\lambda_1},  \frac{2\pi\ell}{\lambda_1}+\frac{\pi}{\lambda_1}].
\]
On this time interval,~$-\frac{a_1}{a_2}\cos(\lambda_1 t)$ is monotonic, and attains all the values from~$-a_1/a_2$ to~$a_1/ a_2$ and since~$a_1>a_2$, we conclude that it intersects~$\cos(\lambda_2t)$ at least once. Also, since~$\lambda_1>\lambda_2$ there is exactly one such intersection, so $m(t)$ has   a single  zero on~$L_1$. A similar argument shows that~$m(t)$ has exactly one zero on 
\[
L_2:=[   \frac{2\pi\ell}{\lambda_1}+\frac{\pi}{\lambda_1}, \frac{2\pi(\ell+1)}{\lambda_1}  ].
\]
Thus, there are two zeros of~$m(t)$ in  every period of~$   \cos(\lambda_1 t)$. This proves~\eqref{eq:a1la2}.

Now consider the case~$a_1<a_2$.  The upper bound on~$N(T)$ in~\eqref{eq:seclow} is obtained arguing similarly as above, but noting that now~$|a_1/a_2|<1$. 
To derive the lower  bound, 
fix an integer~$\ell\geq 0$, and consider the time interval
\[
\tilde L_1:=[\frac{2\pi\ell}{\lambda_2},  \frac{2\pi\ell}{\lambda_2}+\frac{\pi}{\lambda_2}].
\]
On this time interval,~$ \cos(\lambda_2 t)$ is monotonic, and attains all the values in the range~$[-1,1] $.  Since~$-\frac{a_1}{a_2}\cos(\lambda_1t)$ takes values inside the interval~$(-1,1)$, there is at least one zero of~$m(t)$ in~$\tilde L_1$.  Arguing similarly as above on the interval
\[
\tilde L_2:=[  \frac{2\pi\ell}{\lambda_2}+\frac{\pi}{\lambda_2},
\frac{2\pi(\ell+1)}{\lambda_2}  ] 
\]
gives the lower bound on~$N(T)$ in~\eqref{eq:seclow}. 
    \end{IEEEproof}

Not that  Proposition~\ref{prop:two_oscillators} implies that 
    \begin{align*} 
        \frac{\lambda_2 }{\pi}T\leq N(T),  \text{ for all } T>0,
    \end{align*}
  and this agrees with the asymptotic 
  linear bound  in Theorem~\ref{thm:main}.

\section{Discussion}
We studied  the number of switching  points~$N(T)$ in time-optimal controls of a single-input linear system on the interval~$[0,T]$. We showed that when  all the eigenvalues
of~$A$ are purely imaginary 
the number of switching points
is lower-bounded by~$c T$ for large~$T$.
By relating the question to the
solution of
the mean motion problem, we also provided 
 an explicit formula for~$c$ in terms of integrals  of Bessel functions.
To the best of our knowledge, this is the first connection between optimal control theory and the mean motion problem.

An interesting question is whether the
bound can be strengthened to
\begin{equation}\label{asymp_eq2}
N(T)= \tilde c T+o(T), \text{ as } T\to\infty,
\end{equation}
with an  explicit positive parameter~$\tilde c$.
The relation to the mean motion problem suggests that~$\widetilde c=\frac{|\Omega|}{\pi}$, where~$\Omega=\sum \lambda_k V_k  $, and the probabilities $$V_k:=W_{n-1}(a_k;a_1,\dots,a_{k-1},a_{k+1},\dots,a_n)$$ are given by the BWW formula \eqref{BWW}.
Another natural research direction is to extend the analysis to the case where~$A$ has  a more general spectral structure.

 \section*{Appendix: proof of the BWW formula }
Let~$\xi:=\sum_{k=1}^n  a_ke^{i\phi_k}\in \C$.
Fix~$r>0$.
By definition,
\be\label{eq:bydef_W}
W_n(r)=\frac1{(2\pi)^n}\int_{\T} 1_{|\xi|\leq r}(\xi)  \prod_{k=1}^n d\phi_k . 
\ee
Writing
\begin{equation}
1_{|\xi|\leq r}(\xi)=1_{B^2}(\xi/r),
\end{equation}
and applying an inverse Fourier transform to~\eqref{eq:ind_func_B2} yields
\begin{align*}
1_{|\xi|\leq r}(\xi)&=1_{B^2}(\xi/r)\\&=\frac{1}{2\pi}\int_\C |\eta|^{-1} J_1(|\eta|)e^{i\Re  (\bar\eta \xi/r )} d\eta .
\end{align*}
Defining the integration variable
  $\theta:=\eta/r$ gives
\begin{equation}
1_{|\xi|\leq r}(\xi)= \frac{r}{2\pi}\int_\C |\theta|^{-1} J_1(r|\theta|)e^{i\Re  (\bar\theta \xi )} d\theta.
\end{equation}
Substituting  this in~\eqref{eq:bydef_W} yields
\begin{align*}
W_n(r)&=\frac{r}{ 2\pi }\int_{\T}  \int_\C |\theta|^{-1} J_1(r|\theta|)e^{i\Re  (\bar\theta \sum_{k=1}^n a_ke^{i\phi_k}  )} d\theta\\&\times \frac1{(2\pi)^n} \prod_{k=1}^n d\phi_k\\
&= \frac{1}{ 2\pi }\frac{1}{(2\pi)^n}\int_{\T}  \int_\C |\theta|^{-1} J_1(r|\theta|)\\&\times
\left (  \prod_{k=1}^n e^{i\Re  (\bar\theta  a_k e^{i\phi_k}  )}d\phi_k\right)  d\theta  .
\end{align*}
Write~$ \theta$ in the polar representation~$   \theta=\rho e^{i \alpha }$, where $\rho:=|\theta|$. Then
\begin{align*}
 \frac1{2\pi}  \int_0^{2\pi } e^{i\Re  (\bar\theta  a_k e^{i\phi_k}  )}d\phi_k &=\frac1{2\pi}
  \int_0^{2\pi }  e^{i\Re  ( \rho a_k   e^{i(\phi_k-\alpha)}  )}d\phi_k\\
  &=\frac1{2\pi}  \int_0^{2\pi }  e^{i \rho \Re  (  a_k   e^{i \beta_k  }  )}d\beta_k\\
  &= \frac1{2\pi}  \int_0^{2\pi }  e^{i \rho   |a_k|  \cos(\beta_k)  }d\beta_k\\
  &= J_0(\rho |a_k|).
\end{align*}
Thus,
\begin{align*}
W_n(r)
&= \frac{1}{ 2\pi }   \int_\C |\theta|^{-1} J_1(r|\theta|)\left (  \prod_{k=1}^n J_0(|\theta| |a_k|)   \right)  d\theta  .
\end{align*}
 The integrand
 \[
 |\theta|^{-1} J_1(r|\theta|) \prod_{k=1}^n J_0(|\theta| |a_k|)
 \]
 is a radial function, and using known results on the integration of radial functions (see, e.g.,~\cite[Chapter~6]{real_analysis_stromberg})
gives
\begin{align*}
W_n(r)
&=  r \int_0^\infty  J_1(r \rho ) \left (  \prod_{k=1}^n  J_0(\rho |a_k|)    \right)  d\rho,
\end{align*}
and this completes the proof of the BWW formula .



\end{document}